\documentclass[oneside,reqno,american]{amsart}
\usepackage[T1]{fontenc}
\usepackage[utf8]{inputenc}
\usepackage{xcolor}
\usepackage{babel}
\usepackage{prettyref}
\usepackage{amstext}
\usepackage{amsthm}
\usepackage{amssymb}
\usepackage[pdfusetitle,
 bookmarks=true,bookmarksnumbered=false,bookmarksopen=false,
 breaklinks=false,pdfborder={0 0 0},pdfborderstyle={},backref=false,colorlinks=false]
 {hyperref}
\hypersetup{
 colorlinks=true,citecolor=blue,linkcolor=blue,linktocpage=true}

\makeatletter
\numberwithin{equation}{section}
\numberwithin{figure}{section}

\usepackage{prettyref}

\newrefformat{cor}{Corollary~\ref{#1}}
\newrefformat{subsec}{Section~\ref{#1}}
\newrefformat{lem}{Lemma~\ref{#1}}
\newrefformat{thm}{Theorem~\ref{#1}}
\newrefformat{sec}{Section~\ref{#1}}
\newrefformat{chap}{Chapter~\ref{#1}}
\newrefformat{prop}{Proposition~\ref{#1}}
\newrefformat{exa}{Example~\ref{#1}}
\newrefformat{tab}{Table~\ref{#1}}
\newrefformat{rem}{Remark~\ref{#1}}
\newrefformat{def}{Definition~\ref{#1}}
\newrefformat{fig}{Figure~\ref{#1}}
\newrefformat{claim}{Claim~\ref{#1}}

\makeatother

\theoremstyle{plain}
\newtheorem{thm}{\protect\theoremname}[section]
\newtheorem{cor}[thm]{\protect\corollaryname}
\theoremstyle{remark}
\newtheorem{rem}[thm]{\protect\remarkname}
\theoremstyle{plain}
\newtheorem{lem}[thm]{\protect\lemmaname}
\newtheorem{prop}[thm]{\protect\propositionname}
\providecommand{\corollaryname}{Corollary}
\providecommand{\lemmaname}{Lemma}
\providecommand{\propositionname}{Proposition}
\providecommand{\remarkname}{Remark}
\providecommand{\theoremname}{Theorem}

\begin{document}
\title[Alternating Weighted Residual Flows]{Alternating Weighted Residual Flows and the Non-Commutative Gap}
\begin{abstract}
This work develops a nonlinear analogue of alternating projections
on Hilbert space, based on iterating a weighted residual transformation
that removes the portion of an operator detected by a projection after
conjugation by its square root. Although this map is neither linear
nor variational and falls outside classical operator-mean frameworks,
the alternating flow between two fixed projections is shown to be
monotone and to converge strongly to a positive limit supported on
their common kernel. The analysis identifies an intrinsic representation
of this limit inside the operator range of the initial datum, which
makes it possible to compare the nonlinear limit with the shorted
operator of Anderson-Duffin-Trapp. The nonlinear flow always produces
an operator dominated by the shorted operator, with equality precisely
in the commuting regime. A global energy identity describes how mass
is dissipated at each step of the iteration, and a factorized description
localizes the gap between the nonlinear limit and the classical shorted
operator.
\end{abstract}

\author{James Tian}
\address{Mathematical Reviews, 535 W. William St, Suite 210, Ann Arbor, MI
48103, USA}
\email{james.ftian@gmail.com}

\maketitle
\tableofcontents{}

\section{Introduction}\label{sec:1}

Alternating projections and operator compressions form a classical
part of operator theory. Von Neumann's theorem on iterated projections
\cite{MR34514} identifies the strong limit of the alternating product
$P_{B}P_{A}P_{B}\dots$ as the projection onto the intersection of
the ranges, and Halmos's canonical decomposition for two projections
\cite{MR251519} gives a complete geometric description of this convergence
in terms of rigid and generic sectors. A substantial literature extends
and refines this picture through angle criteria, averaged products,
and infinite projection cycles \cite{MR2252935,MR3773065,MR3145756,MR2903120,MR4310540}. 

In the setting of positive operators, linear compressions lie at the
core of the theory of shorted operators developed by Anderson, Duffin,
and Trapp \cite{MR242573,MR287970,MR356949}. Their work interprets
the shorted operator to a subspace $K$ as the largest positive operator
dominated by $R$ and supported on $K$. It connects to operator ranges
and invariance problems studied by Douglas \cite{MR203464} and by
Fillmore-Williams \cite{MR293441}, and continues through spectral
variants such as \cite{MR2234254}. These constructions share a common
feature: they are linear and their asymptotic behavior is governed
by invariance of subspaces and variational characterizations.

The present note investigates a nonlinear analogue of this classical
picture. Given a projection $P$ on a Hilbert space $H$ and a positive
operator $R$, we consider the residual map 
\[
\Phi_{P}\left(R\right)=R^{1/2}\left(I-P\right)R^{1/2},
\]
which removes the portion of $R$ detected by $P$ after conjugation
by the operator square root. This transformation is nonlinear, fails
to preserve affine structure on the cone of positive operators, and
does not arise from a variational principle or from any operator mean
in the sense of Kubo-Ando \cite{MR563399}. Its long-term behavior
is accordingly more delicate.

We study the dynamical system generated by alternating this map between
two fixed projections $P_{A}$ and $P_{B}$. Starting from an initial
positive operator $R_{0}$, we define 
\[
R_{n+1}=\begin{cases}
\Phi_{P_{B}}\left(R_{n}\right) & n\text{ even},\\
\Phi_{P_{A}}\left(R_{n}\right) & n\text{ odd}.
\end{cases}
\]
Our first result (\prettyref{thm:it}) shows that the sequence $\left(R_{n}\right)$
is monotone in the Loewner order and converges strongly to a positive
limit $R_{\infty}$. This limit is supported on the intersection 
\[
K=\ker P_{A}\cap\ker P_{B},
\]
and is a simultaneous fixed point of both residual maps. Thus, despite
the nonlinearity of the transformation, the convergence mechanism
is governed by the same geometric subspace $K$ that appears in the
linear theory.

A natural point of comparison is the shorted operator $S=R_{0}\vert_{K}$,
the maximal positive operator dominated by $R_{0}$ and supported
on $K$. We show that $R_{\infty}\le S$, but equality need not hold:
the nonlinear dynamics can dissipate mass that the classical shorting
construction preserves. Working intrinsically on 
\[
H_{R_{0}}:=\overline{ran\,(R^{1/2}_{0})},\qquad M:=\overline{\{u\in H_{R_{0}}:R^{1/2}_{0}u\in K\}},
\]
we obtain 
\[
S=R^{1/2}_{0}P_{M}R^{1/2}_{0},\qquad R_{\infty}=R^{1/2}_{0}T_{\infty}R^{1/2}_{0},\qquad0\le T_{\infty}\le P_{M}.
\]
Equality $R_{\infty}=S$ holds if and only if $T_{\infty}=P_{M}$
on $H_{R_{0}}$. Equivalently, the defect $G:=P_{M}-T_{\infty}$ vanishes,
and the gap localizes as 
\[
S-R_{\infty}=R^{1/2}_{0}GR^{1/2}_{0}.
\]
A finite-dimensional example shows the discrepancy can be extreme:
one may have $S>0$ while $R_{\infty}=0$. In the commuting regime,
the defect disappears and $R_{\infty}=S$.

To analyze this discrepancy, we pass to the operator--range space
and factor every $0\le R\le R_{0}$ as $R=R^{1/2}_{0}TR^{1/2}_{0}$
with a unique positive contraction $T$. This yields intrinsic residual
maps $T\mapsto\Psi_{P}\left(T\right)$ that are monotone decreasing,
produce the limit $T_{\infty}$, and make the comparison with shorting
transparent. In parallel, a global energy identity decomposes the
initial operator into the limiting piece and a strongly convergent
series of dissipated terms, providing a quantitative account of mass
loss along the iteration.

The remainder of the paper is organized as follows. \prettyref{sec:2}
develops the iteration, proves monotone strong convergence, and identifies
the geometric support of the limit. \prettyref{sec:3} establishes
the global energy identity and interprets the dissipated mass. \prettyref{sec:4}
gives a finite-dimensional example illustrating strict inequality
with the shorted operator. \prettyref{sec:5} builds the intrinsic
operator-range framework, compares the nonlinear limit to the shorted
operator, localizes the gap, and characterizes equality.

\section{Alternating Weighted Residual Flows}\label{sec:2}

This section sets up the alternating weighted-residual iteration and
proves its basic convergence. We fix two orthogonal projections and
a positive initial operator, show the sequence is monotone and converges
strongly to a limit supported on the common kernel, and record the
resulting fixed-point relations. 
\begin{thm}[alternating residual flow]
\label{thm:it}Let $H$ be a Hilbert space, $P_{A},P_{B}$ orthogonal
projections on $H$, and $R_{0}\in B\left(H\right)$ a positive operator.
For any projection $P$ and any positive operator $R$, define the
residual map 
\[
\Phi_{P}\left(R\right):=R^{1/2}\left(I-P\right)R^{1/2}.
\]
Define a sequence $\left(R_{n}\right)_{n\ge0}$ by 
\[
R_{n+1}=\begin{cases}
\Phi_{P_{B}}\left(R_{n}\right), & n\text{ even},\\[4pt]
\Phi_{P_{A}}\left(R_{n}\right), & n\text{ odd},
\end{cases}\quad n\ge0,
\]
starting from $R_{0}$. Let 
\[
K:=\ker P_{A}\cap\ker P_{B}.
\]
Then the sequence $\left(R_{n}\right)$ is decreasing in the Loewner
order: 
\[
0\le R_{n+1}\le R_{n}\le R_{0}\quad\text{for all }n.
\]
In particular, there exists a (unique) positive operator $R_{\infty}\in B\left(H\right)$
such that 
\[
R_{n}\xrightarrow{s}R_{\infty}
\]
in the strong operator topology. 

Moreover, the limit $R_{\infty}$ satisfies: 
\begin{enumerate}
\item $0\le R_{\infty}\le R_{0}$. 
\item $ran\left(R_{\infty}\right)\subset K$. 
\item $R_{\infty}$ is a fixed point of both residual maps: 
\[
R_{\infty}=\Phi_{P_{A}}(R_{\infty})=\Phi_{P_{B}}(R_{\infty}).
\]
\end{enumerate}
That is, the alternating residual flow converges strongly to a positive
operator $R_{\infty}$ bounded by $R_{0}$ and supported inside $\ker P_{A}\cap\ker P_{B}$,
which is invariant under both residual maps. 

\end{thm}

\begin{proof}
Fix an orthogonal projection $P$ and a positive operator $R\in B\left(H\right)$.
For any $x\in H$, 
\[
\left\langle x,\Phi_{P}\left(R\right)x\right\rangle =\Vert\left(I-P\right)R^{1/2}x\Vert^{2}\ge0.
\]
Thus $\Phi_{P}(R)\ge0$. Since $\Phi_{P}\left(R\right)=R-R^{1/2}PR^{1/2}$
and $R^{1/2}PR^{1/2}\geq0$, we get 
\begin{equation}
0\le\Phi_{P}\left(R\right)\le R\label{eq:a-1}
\end{equation}
in the Loewner order. 

By definition, $R_{n+1}=\Phi_{P_{*}}\left(R_{n}\right)$, where $P_{*}=P_{B}$
for $n$ even and $P_{*}=P_{A}$ for $n$ odd. So, by \eqref{eq:a-1},
\[
0\leq R_{n+1}\leq R_{n}\leq R_{0}
\]
for all $n$. We recall a standard fact (see e.g., \cite{MR493419}):
If $\left(S_{n}\right)$ is a bounded decreasing sequence of positive
operators in $B\left(H\right)$, then there exists a unique positive
operator $S$ such that $S_{n}\xrightarrow{s}S$. Applying this to
$S_{n}=R_{n}$, we obtain a unique positive operator $R_{\infty}$
such that $R_{n}\xrightarrow{s}R_{\infty}.$ Since the Loewner order
is closed under strong limits for decreasing sequences, we have $0\le R_{\infty}\le R_{0}$. 

For the fixed point property, it is convenient to look at even and
odd indices separately. Define 
\[
T:=\Phi_{P_{A}}\circ\Phi_{P_{B}}.
\]
Then for each $n$, $R_{2n}=T^{n}\left(R_{0}\right)$. The sequence
$\left(R_{2n}\right)$ is positive, bounded, and decreasing, hence
it converges strongly to some $R_{\mathrm{even}}\ge0$, i.e., $R_{2n}\xrightarrow{s}R_{\mathrm{even}}$.
Similarly, $R_{2n+1}\xrightarrow{s}R_{\mathrm{odd}}$ for some $R_{\mathrm{odd}}\ge0$.
By construction, $R_{2n+1}=\Phi_{P_{B}}\left(R_{2n}\right)$. 

We now show that $R_{\mathrm{odd}}=\Phi_{P_{B}}\left(R_{\mathrm{even}}\right)$.
For this, we need the continuity of the square root map on such monotone
sequences. A standard result says: If $S_{n}$ is a bounded decreasing
sequence of positive operators that converges strongly to $S$, then
$S^{1/2}_{n}$ also decreases and converges strongly to $S^{1/2}$.

Apply this to $S_{n}=R_{2n}$ and $S=R_{\mathrm{even}}$. Then $R^{1/2}_{2n}\xrightarrow{s}R^{1/2}_{\mathrm{even}}$.
Fix any $x\in H$. Then 
\[
\left\langle x,R_{2n+1}x\right\rangle =\left\langle x,\Phi_{P_{B}}\left(R_{2n}\right)x\right\rangle =\Vert\left(I-P_{B}\right)R^{1/2}_{2n}x\Vert^{2}.
\]
Since $I-P_{B}$ is bounded and $R^{1/2}_{2n}x\to R^{1/2}_{\mathrm{even}}x$,
we have 
\[
\left(I-P_{B}\right)R^{1/2}_{2n}x\to\left(I-P_{B}\right)R^{1/2}_{\mathrm{even}}x
\]
in norm. Thus 
\[
\Vert\left(I-P_{B}\right)R^{1/2}_{2n}x\Vert^{2}\to\Vert\left(I-P_{B}\right)R^{1/2}_{\mathrm{even}}x\Vert^{2}.
\]
The left-hand $\left\langle x,R_{2n+1}x\right\rangle $ converges
to $\left\langle x,R_{\mathrm{odd}}x\right\rangle $. Therefore 
\begin{align*}
\left\langle x,R_{\mathrm{odd}}x\right\rangle  & =\Vert\left(I-P_{B}\right)R^{1/2}_{\mathrm{even}}x\Vert^{2}\\
 & =\langle x,R^{1/2}_{\mathrm{even}}\left(I-P_{B}\right)R^{1/2}_{\mathrm{even}}x\rangle=\langle x,\Phi_{P_{B}}\left(R_{\mathrm{even}}\right)x\rangle
\end{align*}
for all $x$. By polarization and boundedness, this implies 
\[
R_{\mathrm{odd}}=\Phi_{P_{B}}\left(R_{\mathrm{even}}\right).
\]
Similarly, since $R_{2n+2}=\Phi_{P_{A}}\left(R_{2n+1}\right)$, and
$R_{2n+1}\xrightarrow{s}R_{\mathrm{odd}}$, the same reasoning yields
\[
R_{\mathrm{even}}=\Phi_{P_{A}}\left(R_{\mathrm{odd}}\right).
\]

On the other hand, the inequalities $R_{2n}\ge R_{2n+1}$ for all
$n$ imply $R_{\mathrm{even}}\ge R_{\mathrm{odd}}$. Thus, by \eqref{eq:a-1},
\[
\Phi_{P_{A}}(R_{\mathrm{odd}})\le R_{\mathrm{odd}}.
\]
But we have just shown $\Phi_{P_{A}}(R_{\mathrm{odd}})=R_{\mathrm{even}}$,
so $R_{\mathrm{even}}\le R_{\mathrm{odd}}$. This shows that $R_{\mathrm{even}}=R_{\mathrm{odd}}=:R_{\infty}$.
Moreover, $R_{\infty}=\Phi_{P_{B}}\left(R_{\infty}\right)=\Phi_{P_{A}}\left(R_{\infty}\right)$. 

Fix any projection $P$ and positive operator $R$ satisfying $R=R^{1/2}\left(I-P\right)R^{1/2}$.
Then 
\[
0=R-R^{1/2}\left(I-P\right)R^{1/2}=R^{1/2}PR^{1/2}.
\]
For any $x\in H$, 
\[
\langle x,R^{1/2}PR^{1/2}x\rangle=\langle PR^{1/2}x,PR^{1/2}x\rangle=\Vert PR^{1/2}x\Vert^{2},
\]
so $\Vert PR^{1/2}x\Vert^{2}=0$ for all $x$, hence $PR^{1/2}=0$.
Therefore $ran\,(R^{1/2})\subset\ker P$, and consequently 
\[
ran\left(R\right)\subset ran\,(R^{1/2})\subset\ker P.
\]
Apply this with $R=R_{\infty}$ and $P=P_{A}$, using $R_{\infty}=\Phi_{P_{A}}(R_{\infty})$.
We get $\mathrm{Ran}(R_{\infty})\subset\ker P_{A}$. Likewise, using
$R_{\infty}=\Phi_{P_{B}}(R_{\infty})$, $ran\,(R_{\infty})\subset\ker P_{B}$.
Thus $ran\,(R_{\infty})\subset\ker P_{A}\cap\ker P_{B}=K$.
\end{proof}
The corollary below identifies the entire set of attainable limits
of the alternating WR flow: the dynamics converge precisely to the
positive cone $B\left(K\right)_{+}$.
\begin{cor}
$\left\{ R_{\infty}\left(R\right):R\in B\left(H\right)_{+}\right\} =B\left(K\right)_{+}$. 
\end{cor}

\begin{proof}
For any $R_{0}\in B\left(H\right)_{+}$, let $R_{\infty}\left(R\right)$
be its corresponding fixed point under the alternating WR flow. By
the theorem, $ran\left(R_{\infty}\left(R\right)\right)\subset K$,
and so $R_{\infty}\left(R\right)\in B\left(K\right)_{+}$. 

Conversely, let $T\in B\left(K\right)_{+}$. Since $\text{Ran}\left(T\right)\subset K$
and $K$ is closed, we have 
\[
ran\,(T^{1/2})=\overline{ran\left(T\right)}\subset K.
\]
Since $K=\ker P_{A}\cap\ker P_{B}$, this means $P_{A}T^{1/2}=P_{B}T^{1/2}=0$,
therefore $\Phi_{P_{A}}\left(T\right)=\Phi_{P_{B}}\left(T\right)=T$,
and $R_{\infty}\left(T\right)=T$\@.
\end{proof}

\section{Energy Decomposition and Dissipation}\label{sec:3}

In this section we record a global identity for the alternating WR
iteration which expresses the initial operator $R_{0}$ as the sum
of the limit $R_{\infty}$ and the cumulative ``dissipated'' components
extracted at each step. This provides a constructive resolution of
the case $R_{\infty}=0$ and clarifies how mass is removed from the
iteration by the alternating projections.

Throughout, let $P_{n}=P_{B}$ for $n$ even and $P_{n}=P_{A}$ for
$n$ odd, so that $R_{n+1}=\Phi_{P_{n}}\left(R_{n}\right)$.
\begin{thm}
\label{thm:4-1}Let $\left(R_{n}\right)_{n\ge0}$ be the alternating
WR sequence generated by $R_{0}\ge0$. Then 
\[
R_{0}=R_{\infty}+\sum^{\infty}_{n=0}R^{1/2}_{n}P_{n}R^{1/2}_{n},
\]
where each $D_{n}:=R^{1/2}_{n}P_{n}R^{1/2}_{n}$ is positive and the
series converges in the strong operator topology. 
\end{thm}

\begin{proof}
By definition of the WR update, 
\[
R_{n+1}=R^{1/2}_{n}\left(I-P_{n}\right)R^{1/2}_{n}=R_{n}-R^{1/2}_{n}P_{n}R^{1/2}_{n}.
\]
Thus 
\[
R_{n}-R_{n+1}=R^{1/2}_{n}P_{n}R^{1/2}_{n}=:D_{n}\ge0.
\]
Summing from $n=0$ to $N$ gives the telescoping identity 
\[
\sum^{N}_{n=0}D_{n}=\sum^{N}_{n=0}\left(R_{n}-R_{n+1}\right)=R_{0}-R_{N+1}.
\]
By \prettyref{thm:it}, $R_{N+1}\to R_{\infty}$ strongly. Since the
partial sums $\sum^{N}_{n=0}D_{n}$ form an increasing net of positive
operators bounded above by $R_{0}$, they converge strongly to a positive
operator $D$, and passing to the strong limit in the identity above
yields 
\[
D=R_{0}-R_{\infty}.
\]
Hence 
\[
R_{0}=R_{\infty}+\sum^{\infty}_{n=0}D_{n}=R_{\infty}+\sum^{\infty}_{n=0}R^{1/2}_{n}P_{n}R^{1/2}_{n},
\]
with strong convergence. Uniqueness is immediate from the deterministic
definition of $\left(R_{n}\right)$. 
\end{proof}
\begin{cor}
\label{cor:4-2}The limit satisfies $R_{\infty}=0$ if and only if
\[
R_{0}=\sum^{\infty}_{n=0}R^{1/2}_{n}P_{n}R^{1/2}_{n}\quad\text{ (strongly)}.
\]
\end{cor}

\begin{rem}
The term $D_{n}=R^{1/2}_{n}P_{n}R^{1/2}_{n}$ is precisely the component
of $R_{n}$ lying in the range of the projection selected at step
$n$. The decomposition of \prettyref{thm:4-1} therefore expresses
the evolution as a global balance: 
\[
\text{initial operator}=\text{surviving component on }K+\text{total dissipated mass}.
\]
In particular, the dissipated operator $\sum_{n\ge0}D_{n}=R_{0}-R_{\infty}$
aggregates all leakage through $P_{A}$ and $P_{B}$ over the entire
iteration.
\end{rem}

\section{A Finite-Dimensional Example}\label{sec:4}

We now give a concrete $2\times2$ example showing that the limiting
operator $R_{\infty}$ from the WR iteration need not coincide with
the maximal positive operator under $R_{0}$ supported on $K=\ker P_{A}\cap\ker P_{B}$.
In particular, $R_{\infty}$ may be strictly smaller (even zero) while
the maximal such operator is nonzero.

Let $H=\mathbb{R}^{2}$ with standard basis $\{e_{1},e_{2}\}$. Let
\[
R_{0}=\begin{pmatrix}5 & 3\\[2pt]
3 & 2
\end{pmatrix}.
\]
Define $P_{A}=\left|e_{1}\left\rangle \right\langle e_{1}\right|$
and $P_{B}=0$, so that 
\[
\ker P_{A}=\mathrm{span}\{e_{2}\},\quad\ker P_{B}=\mathbb{R}^{2},
\]
and 
\[
K=\ker P_{A}\cap\ker P_{B}=\mathrm{span}\{e_{2}\}.
\]
We choose $P_{B}=0$ so that $\Phi_{P_{B}}$ is the identity, and
the alternating iteration reduces to repeated application of $\Phi_{P_{A}}$.

\subsection*{Shorted/Schur complement-type maximal operator on $K$}

We first compute the maximal positive operator $X$ with: 
\[
0\le X\le R_{0},\quad ran\left(X\right)\subset K=\mathrm{span}\{e_{2}\}.
\]
Any positive operator $X$ supported on $\mathrm{span}\{e_{2}\}$
has the form 
\[
X=\begin{pmatrix}0 & 0\\[2pt]
0 & t
\end{pmatrix},\quad t\ge0.
\]
The inequality $X\le R_{0}$ is equivalent to $R_{0}-X\ge0$, i.e.,
\[
R_{0}-X=\begin{pmatrix}5 & 3\\[2pt]
3 & 2-t
\end{pmatrix}\geq0.
\]
Equivalently, 
\[
\det(R_{0}-X)=1-5t\geq0\Longleftrightarrow t\le0.2.
\]
Therefore the maximal possible choice is
\[
X=\begin{pmatrix}0 & 0\\[2pt]
0 & 0.2
\end{pmatrix}.
\]
One can check directly that $R_{0}-X$ is positive semidefinite: 
\[
R_{0}-X=\begin{pmatrix}5 & 3\\[2pt]
3 & 1.8
\end{pmatrix}
\]
with $\det=0$. 

\subsection*{WR iteration and its limit}

Now we compute the WR sequence $\left(R_{n}\right)$ for this data.
Since $P_{B}=0$, the map $\Phi_{P_{B}}$ is the identity: 
\[
\Phi_{P_{B}}(R)=R^{1/2}(I-0)R^{1/2}=R.
\]
Thus the alternating scheme reduces to: 
\[
R_{n+1}=\Phi_{P_{A}}(R_{n}),\quad R_{0}\ \text{given}.
\]
To make the calculations more explicit, note that 
\[
R^{1/2}_{0}=\begin{pmatrix}2 & 1\\[2pt]
1 & 1
\end{pmatrix}.
\]
Then we have 
\[
R_{1}=\Phi_{P_{A}}\left(R\right)=R^{1/2}P^{\perp}_{A}R^{1/2}=\begin{pmatrix}1 & 1\\[2pt]
1 & 1
\end{pmatrix}.
\]

Observe that $R_{1}=2P_{v}$, where 
\[
v=\frac{1}{\sqrt{2}}\begin{pmatrix}1\\[2pt]
1
\end{pmatrix},\quad P_{v}=\left|v\left\rangle \right\langle v\right|=\frac{1}{2}\begin{pmatrix}1 & 1\\[2pt]
1 & 1
\end{pmatrix}.
\]
Hence, 
\[
R^{1/2}_{1}=\sqrt{2}P_{v}.
\]
It follows that 
\[
R_{2}=\Phi_{P_{A}}\left(R_{1}\right)=R^{1/2}_{1}P^{\perp}_{A}R^{1/2}_{1}=\frac{1}{2}R_{1}.
\]

By induction, for $n\ge1$, 
\[
R_{n}=2^{-(n-1)}R_{1}\xrightarrow[n\rightarrow\infty]{}R_{\infty}=0.
\]

\subsection{Discussion}

The example above shows a phenomenon that contrasts with standard
results in the operator-theoretic literature: the nonlinear residual
transformation 
\begin{equation}
\Phi_{P}\left(R\right)=R^{1/2}\left(I-P\right)R^{1/2},\label{eq:b-1}
\end{equation}
when iterated alternately along two projections, converges to a positive
operator supported in $\ker P_{A}\cap\ker P_{B}$, yet this limit
is generally strictly smaller than the maximal operator satisfying
this support constraint. 

1. For a positive operator $R_{0}$ and a subspace $K$, the shorted
operator $R_{0}|_{K}$ (also understood as a generalized Schur complement)
is the largest positive operator $X\le R_{0}$ with $ran\left(X\right)\subseteq K$.
This notion was introduced and developed in depth by Anderson, Duffin,
and Trapp \cite{MR242573,MR287970,MR356949}; see also \cite{MR24575,MR24574}
and \cite{MR203464}. If $P$ is the projection onto $K^{\perp}$,
the shorted operator admits the variational characterization: 
\[
\left\langle x,R_{0}|_{K}x\right\rangle =\inf_{y\in K^{\perp}}\left\langle x+y,R_{0}\left(x+y\right)\right\rangle ,\quad x\in K.
\]
The 2D computation above shows that the operator produced by WR iteration
may be strictly smaller than $R_{0}|_{K}$.

2. Classical compressions of the form 
\[
R\longmapsto(I-P)R(I-P)
\]
are linear and order-preserving. They arise in settings such as Halmos's
work on two projections \cite{MR251519} and the theory of alternating
projections. The WR map \eqref{eq:b-1} is neither linear nor order-preserving,
and it coincides with compression only in the commuting case $[R,P]=0$.
As the example illustrates, this can result in the iterative suppression
of the operator's magnitude even in directions orthogonal to $\mathrm{Ran}(P)$.
This behavior is distinct from classical alternating projection algorithms
(von Neumann, Halmos) and modern extensions in convex optimization,
which typically rely on linearity, norm-contractivity, or monotonicity
in the Loewner order.

3. A broad class of nonlinear operator transformations is captured
by the Kubo-Ando theory of operator means \cite{MR563399}, where
each mean is monotone, jointly concave, and order-preserving. The
map \eqref{eq:b-1} does not constitute an operator mean in this sense,
as it fails monotonicity, concavity, and the transformer inequality. 

4. Nonlinear dynamical systems on positive operators, such as power
means, Riccati flows, and fixed-point iterations of completely positive
maps, tend to rely on order-theoretic structure or convexity. The
WR map operates outside these constraints. 

In summary, the operator $R_{\infty}$ defines a specific construction
associated with the triple $(R_{0},P_{A},P_{B})$, arising from a
nonlinear residual dynamic rather than from variational principles.
This suggests several directions for further inquiry, such as quantitative
comparisons with the shorted operator, and characterization of the
commuting regime. 

\section{Comparison: $R_{\infty}$ vs the shorted operator $R_{0}|_{K}$}\label{sec:5}

We keep the notation from \prettyref{thm:it}: $H$ is a Hilbert space,
$P_{A},P_{B}$ orthogonal projections on $H$, $R_{0}\in B\left(H\right)$
positive, and 
\[
\Phi_{P}\left(R\right):=R^{1/2}\left(I-P\right)R^{1/2},\qquad R_{n+1}=\begin{cases}
\Phi_{P_{B}}\left(R_{n}\right), & n{\rm ~even},\\
\Phi_{P_{A}}\left(R_{n}\right), & n{\rm ~odd},
\end{cases}
\]
with $R_{0}$ given. We write 
\[
K:=\ker\left(P_{A}\right)\cap\ker\left(P_{B}\right)
\]
and let $P_{K}$ be the orthogonal projection onto $K$. 

By \prettyref{thm:it}, $0\le R_{n+1}\le R_{n}\le R_{0}$ for all
$n$, hence $R_{n}\stackrel{s}{\longrightarrow}R_{\infty}\ge0$, with
$R_{\infty}H\subset K$ and $R_{\infty}=\Phi_{P_{A}}\left(R_{\infty}\right)=\Phi_{P_{B}}\left(R_{\infty}\right)$.
We denote by 
\[
S:=R_{0}|_{K}
\]
the shorted operator of $R_{0}$ to $K$. 

A key structural fact about the WR iteration is that the limit $R_{\infty}$
(by construction) satisfies
\begin{equation}
0\le R_{\infty}\le S\leq R_{0}.\label{eq:f-1}
\end{equation}
This allows us to transfer square-root range information from $R_{0}$
to $R_{\infty}$. We use the classical Douglas factorization in the
form below.

Set $H_{R_{0}}:=\overline{R^{1/2}_{0}H}\subset H$. We recall the
standard relation $\ker R^{1/2}_{0}=\ker R_{0}=(H_{R_{0}})^{\perp}$,
which implies that the restriction of $R^{1/2}_{0}$ to $H_{R_{0}}$
is injective.
\begin{lem}
\label{lem:5-2}Let $A,B\in B\left(H\right)_{+}$ with $0\le A\le B$.
Then there exists a bounded operator 
\[
X:\overline{ran\,(B^{1/2})}\to\overline{ran\,(A^{1/2})}
\]
with $\left\Vert X\right\Vert \le1$ such that $A^{1/2}=XB^{1/2}$
on $H$. Consequently, 
\[
ran\,(A^{1/2})\subset ran\,(B^{1/2}).
\]
\end{lem}

\begin{proof}
For details, see \cite{MR203464}. A proof sketch is included below
for completeness.

Define $X$ on $ran\,(B^{1/2})$ by 
\[
X(B^{1/2}x):=A^{1/2}x.
\]
If $B^{1/2}x=B^{1/2}y$, then $B^{1/2}\left(x-y\right)=0$, and 
\[
0\le\left\langle x-y,A\left(x-y\right)\right\rangle \le\left\langle x-y,B\left(x-y\right)\right\rangle =0,
\]
so $A^{1/2}\left(x-y\right)=0$. Thus $A^{1/2}x=A^{1/2}y$. Thus,
$X$ is well defined. 

For all $x\in H$, 
\[
\Vert X(B^{1/2}x)\Vert^{2}=\Vert A^{1/2}x\Vert^{2}=\left\langle x,Ax\right\rangle \le\left\langle x,Bx\right\rangle =\Vert B^{1/2}x\Vert^{2},
\]
so $\left\Vert X\right\Vert \le1$ on $ran\,(B^{1/2})$. Therefore
$X$ extends uniquely to a contraction on $\overline{ran\,(B^{1/2})}$. 

Since $A^{1/2}$ is self-adjoint, 
\[
A^{1/2}=(A^{1/2})^{*}=(XB^{1/2})^{*}=B^{1/2}X^{*}.
\]
Hence $ran\,(A^{1/2})\subset ran\,(B^{1/2})$. 
\end{proof}
\begin{cor}
\label{cor:f-3}We have $ran(R^{1/2}_{\infty})\subset ran(R^{1/2}_{0})\subset H_{R_{0}}$,
and a contraction $X:H_{R_{0}}\to H_{R_{0}}$ such that $R^{1/2}_{\infty}=XR^{1/2}_{0}$.
\end{cor}

\begin{rem}
From $0\le A\le B$ we always have $\overline{ran\,(A^{1/2})}\subset\overline{ran\,(B^{1/2})}$.
\prettyref{lem:5-2} upgrades this to $ran\,(A^{1/2})\subset ran\,(B^{1/2})$
via a contraction factorization $A^{1/2}=XB^{1/2}$. We use this to
work canonically on $H_{B}=\overline{ran\,(B^{1/2})}$ and to represent
$A=B^{1/2}TB^{1/2}$ with $0\le T\le I$.
\end{rem}

\begin{prop}[Factorized comparison]
\label{prop:f-4} Let $H$ be a Hilbert space, $R_{0}\in B(H)$ be
positive, $P_{A},P_{B}$ orthogonal projections, $K=\ker P_{A}\cap\ker P_{B}$,
and $S=R_{0}|_{K}$ the shorted operator of $R_{0}$ to $K$. Set
$H_{R_{0}}=\overline{ran\,(R^{1/2}_{0})}$. Define the subspace 
\begin{equation}
M=\overline{\{u\in H_{R_{0}}:R^{1/2}_{0}u\in K\}}\subset H_{R_{0}}.\label{eq:f-2}
\end{equation}
Then:
\begin{enumerate}
\item \label{enu:5-4-1}For every $R\in B(H)$ with $0\le R\le R_{0}$,
there exists a unique positive contraction $T\in B(H_{R_{0}})$ such
that 
\[
R=R^{1/2}_{0}TR^{1/2}_{0}.
\]
\item \label{enu:5-4-2}The shorted operator admits the intrinsic form 
\[
S=R^{1/2}_{0}P_{M}R^{1/2}_{0},
\]
where $P_{M}$ is the orthogonal projection of $H_{R_{0}}$ onto $M$. 
\item \label{enu:5-4-3}If $R_{\infty}$ is the limit of the alternating
weighted-residual iteration (\prettyref{thm:it}), then 
\[
R_{\infty}=R^{1/2}_{0}T_{\infty}R^{1/2}_{0}\quad\text{with}\quad0\le T_{\infty}\le P_{M}\text{ on }H_{R_{0}}.
\]
Consequently $R_{\infty}\le S$. 
\item \label{enu:5-4-4}Equality holds if and only if the intrinsic contraction
saturates the projector, i.e., 
\[
R_{\infty}=S\iff T_{\infty}=P_{M}\text{ on }H_{R_{0}}.
\]
\end{enumerate}
\end{prop}

\begin{proof}
Fix $R$ with $0\le R\le R_{0}$. By \prettyref{lem:5-2}, there exists
a contraction $X:H_{R_{0}}\rightarrow H_{R_{0}}$ with $R^{1/2}=XR^{1/2}_{0}$.
Set $T:=X^{*}X$ acting on $H_{R_{0}}$. Then $T\ge0$ and $\|T\|\le\|X\|^{2}\le1$.
For any $x\in H$, 
\[
\left\langle x,Rx\right\rangle =\Vert R^{1/2}x\Vert^{2}=\Vert XR^{1/2}_{0}x\Vert^{2}=\langle R^{1/2}_{0}x,TR^{1/2}_{0}x\rangle.
\]
Uniqueness follows from the density of $ran\,(R^{1/2}_{0})$ in $H_{R_{0}}$.
This proves \eqref{enu:5-4-1}.

Let $S':=R^{1/2}_{0}P_{M}R^{1/2}_{0}$. Since $P_{M}\le I$, we have
$0\le S'\le R_{0}$. Moreover, if $x\in H$ and $u=R^{1/2}_{0}x\in H_{R_{0}}$,
then $S'x=R^{1/2}_{0}P_{M}u$, and the definition of $M$ ensures
that $R^{1/2}_{0}P_{M}u\in K$. Thus $ran\left(S'\right)\subset K$,
and $S'$ is admissible for the shorting problem. 

On the other hand, if $R$ is any positive operator with $0\le R\le R_{0}$
and $ran\left(R\right)\subset K$, part \eqref{enu:5-4-1} gives a
unique positive contraction $T\in B\left(H_{R_{0}}\right)$ such that
$R=R^{1/2}_{0}TR^{1/2}_{0}$. The range condition $ran\left(R\right)\subset K$
implies $R^{1/2}_{0}Tu\in K$ for all $u\in H_{R_{0}}$, hence $Tu\in M$,
and therefore 
\[
ran\left(T\right)\subset M,\qquad T=P_{M}TP_{M}.
\]
Since $T$ is a contraction, it follows that $0\le T\le P_{M}$. For
all $x\in H$, 
\[
\left\langle x,Rx\right\rangle =\langle R^{1/2}_{0}x,TR^{1/2}_{0}x\rangle\le\langle R^{1/2}_{0}x,P_{M}R^{1/2}_{0}x\rangle=\left\langle x,S'x\right\rangle ,
\]
so $R\le S'$. As the shorted operator $S=R_{0}|_{K}$ is the maximal
such $R$, we conclude $S=S'$. This is part \eqref{enu:5-4-2}. 

From \prettyref{thm:it}, $0\le R_{\infty}\le R_{0}$ and $ran\left(R_{\infty}\right)\subset K$.
By \eqref{enu:5-4-1}, $R_{\infty}=R^{1/2}_{0}T_{\infty}R^{1/2}_{0}$
for a unique positive contraction $T_{\infty}$. Since $ran\left(R_{\infty}\right)\subset K$,
$T_{\infty}$ maps the dense subspace $ran\,(R^{1/2}_{0})$ into $M$.
By continuity, $ran\left(T_{\infty}\right)\subset M$. Since $T_{\infty}\ge0$
and $ran\left(T_{\infty}\right)\subset M$, we have $T_{\infty}=P_{M}T_{\infty}P_{M}$.
Finally, since $\|T_{\infty}\|\le1$, we have $T_{\infty}\le P_{M}$.
Parts \eqref{enu:5-4-3} and \eqref{enu:5-4-4} follow from this.
\end{proof}
\begin{cor}
With the notation above, the following are equivalent:
\begin{enumerate}
\item \label{enu:5-5-1}$S=R_{0}|_{K}=0$. 
\item \label{enu:5-5-2}$M=\left\{ 0\right\} $, equivalently $P_{M}=0$
on $H_{R_{0}}$. 
\item \label{enu:5-5-3}There is no nonzero vector $u\in H_{R_{0}}$ such
that $R^{1/2}_{0}u$ lies in $K={\rm ker}P_{A}\cap{\rm ker}P_{B}$. 
\end{enumerate}
In this case $R_{\infty}=0$. In general, the converse implication
$R_{\infty}=0\Rightarrow S=0$ need not hold.

\end{cor}

\begin{proof}
Recall that $S=R_{0}|_{K}$ admits the intrinsic form 
\[
S=R^{1/2}_{0}P_{M}R^{1/2}_{0},
\]
where $M\subset H_{R_{0}}$ is the closed subspace in \eqref{eq:f-2}
and $P_{M}$ is the orthogonal projection onto $M$. Since the restriction
of $R^{1/2}_{0}$ to $H_{R_{0}}$ is injective, we have 
\[
S=0\quad\Longleftrightarrow\quad P_{M}=0\quad\Longleftrightarrow\quad M=\left\{ 0\right\} ,
\]
which gives the equivalence of \eqref{enu:5-5-1} and \eqref{enu:5-5-2}.

For the equivalence of \eqref{enu:5-5-2} and \eqref{enu:5-5-3},
note that by definition $M=\left\{ 0\right\} $ holds if and only
if the only vector $u\in H_{R_{0}}$ with $R^{1/2}_{0}u\in K$ is
$u=0$, which is exactly condition \eqref{enu:5-5-3}.

Finally, if $S=0$, then $S$ is the maximal positive operator dominated
by $R_{0}$ with range in $K$, so every such operator must vanish.
In particular, $0\le R_{\infty}\le S=0$ implies $R_{\infty}=0$.
On the other hand, the finite-dimensional example in \prettyref{sec:4}
shows that $R_{\infty}=0$ while $S\neq0$ may occur, so the converse
implication does not hold in general. 
\end{proof}
To express the dynamics intrinsically, we use the factorization $R=R^{1/2}_{0}TR^{1/2}_{0}$
from \prettyref{prop:f-4} to pull the residual maps back to the subspace
$H_{R_{0}}$. This induces a corresponding sequence of contractions
starting from the identity, which encodes the convergence of the original
flow. 
\begin{prop}[Intrinsic residual maps]
 For $P\in\left\{ P_{A},P_{B}\right\} $ and $T\in B\left(H_{R_{0}}\right)$
with $0\le T\le I$, define $\Psi_{P}\left(T\right)$ by 
\[
\Phi_{P}(R^{1/2}_{0}TR^{1/2}_{0})=R^{1/2}_{0}\Psi_{P}\left(T\right)R^{1/2}_{0},
\]
where $\Phi_{P}\left(R\right)=R^{1/2}\left(I-P\right)R^{1/2}$ is
the weighted-residual map. Then:
\begin{enumerate}
\item \label{enu:5-6-1}$0\le\Psi_{P}\left(T\right)\le T$ for all $0\le T\le I$. 
\item \label{enu:5-6-2}If $0\le T\le P_{M}$, then $\Psi_{P}\left(T\right)=T$. 
\end{enumerate}
In particular, if we set $T_{0}=I$ on $H_{R_{0}}$ and define 
\[
T_{n+1}=\begin{cases}
\Psi_{P_{B}}\left(T_{n}\right), & n\text{ even},\\
\Psi_{P_{A}}\left(T_{n}\right), & n\text{ odd},
\end{cases}
\]
then $\left(T_{n}\right)$ is decreasing in the Loewner order, $0\le T_{n}\le I$
for all $n$, and $T_{n}\stackrel{s}{\longrightarrow}T_{\infty}$
for a positive contraction $T_{\infty}$ on $H_{R_{0}}$ with 
\[
R_{\infty}=R^{1/2}_{0}T_{\infty}R^{1/2}_{0}.
\]

\end{prop}

\begin{proof}
Let $T$ satisfy $0\le T\le I$ and set $R=R^{1/2}_{0}TR^{1/2}_{0}$.
Then $0\le R\le R_{0}$. The weighted-residual map gives 
\[
R'=\Phi_{P}\left(R\right)=R^{1/2}\left(I-P\right)R^{1/2}\ge0,\qquad R'\le R\le R_{0}.
\]
By \prettyref{prop:f-4}, there exists a unique positive contraction
$\Psi_{P}\left(T\right)$ on $H_{R_{0}}$ such that $R'=R^{1/2}_{0}\Psi_{P}\left(T\right)R^{1/2}_{0}$.
Since $R'\le R$, injectivity of $R^{1/2}_{0}$ on $H_{R_{0}}$ implies
$\Psi_{P}\left(T\right)\le T$, and clearly $0\le\Psi_{P}\left(T\right)\le I$.
This proves \eqref{enu:5-6-1}.

If $0\le T\le P_{M}$, then $ran\left(T\right)\subset M$ by \prettyref{lem:5-2}.
Hence $ran\left(R\right)=ran\,(R^{1/2}_{0}TR^{1/2}_{0})\subset K$
by the definition of $M$. On $K$ we have $\left(I-P\right)$ equal
to the identity, so 
\[
\Phi_{P}\left(R\right)=R^{1/2}\left(I-P\right)R^{1/2}=R.
\]
By the intrinsic factorization, this is equivalent to $\Psi_{P}\left(T\right)=T$.
This proves \eqref{enu:5-6-2}.

For the iteration, define $T_{0}=I$ and $T_{n+1}$ as above. By \eqref{enu:5-6-1}
we have $0\le T_{n+1}\le T_{n}\le I$ for all $n$, so $\left(T_{n}\right)$
is a bounded decreasing sequence of positive operators on $H_{R_{0}}$.
It follows that $T_{n}\stackrel{s}{\longrightarrow}T_{\infty}$ for
some positive contraction $T_{\infty}$. Setting $R_{n}=R^{1/2}_{0}T_{n}R^{1/2}_{0}$,
the defining relation for $\Psi_{P}$ shows that $\left(R_{n}\right)$
coincides with the alternating weighted-residual sequence from \prettyref{thm:it},
so $R_{n}\stackrel{s}{\longrightarrow}R_{\infty}$. Passing to the
limit in $R_{n}=R^{1/2}_{0}T_{n}R^{1/2}_{0}$ yields 
\[
R_{\infty}=R^{1/2}_{0}T_{\infty}R^{1/2}_{0},
\]
as claimed. 
\end{proof}
\begin{prop}[Gap localization]
 \label{prop:5-7} Set 
\[
G:=P_{M}-T_{\infty}\quad\text{on }H_{R_{0}}.
\]
Then $G\ge0$, $ran\left(G\right)\subset M$, and 
\[
S-R_{\infty}=R^{1/2}_{0}GR^{1/2}_{0}.
\]
In particular, $S=R_{\infty}$ if and only if $G=0$. 
\end{prop}

\begin{proof}
By the factorized comparison in \prettyref{prop:f-4}, we have $0\le T_{\infty}\le P_{M}$
on $H_{R_{0}}$. Hence $G=P_{M}-T_{\infty}\ge0$ and $ran\left(G\right)\subset ran\left(P_{M}\right)=M$.
Using the intrinsic factorizations 
\[
S=R^{1/2}_{0}P_{M}R^{1/2}_{0},\qquad R_{\infty}=R^{1/2}_{0}T_{\infty}R^{1/2}_{0},
\]
we obtain 
\[
S-R_{\infty}=R^{1/2}_{0}\left(P_{M}-T_{\infty}\right)R^{1/2}_{0}=R^{1/2}_{0}GR^{1/2}_{0}.
\]
If $S=R_{\infty}$, then $R^{1/2}_{0}GR^{1/2}_{0}=0$. For any $x\in H$,
\[
\langle R^{1/2}_{0}x,GR^{1/2}_{0}x\rangle=\langle x,R^{1/2}_{0}GR^{1/2}_{0}x\rangle=0.
\]
Thus $\left\langle y,Gy\right\rangle =0$ for all $y$ in the dense
subspace $ran\,(R^{1/2}_{0})\subset H_{R_{0}}$. By continuity and
the fact that $G\ge0$, this extends to all $y\in H_{R_{0}}$, so
$Gy=0$ for every $y\in H_{R_{0}}$, i.e. $G=0$. The converse is
immediate from $S-R_{\infty}=R^{1/2}_{0}GR^{1/2}_{0}$. 
\end{proof}
\begin{cor}[Kernel and support comparison]
 \label{cor:5-8} We have $\ker S\subset\ker R_{\infty}$ and $s\left(R_{\infty}\right)\le s\left(S\right)$,
where $s\left(X\right)$ denotes the support projection of a positive
operator $X$. 
\end{cor}

\begin{proof}
From $0\le R_{\infty}\le S$ we get, for any $x\in H$, 
\[
\left\langle x,Sx\right\rangle =0\implies\left\langle x,R_{\infty}x\right\rangle \le\left\langle x,Sx\right\rangle =0\implies\left\langle x,R_{\infty}x\right\rangle =0,
\]
so $\ker S\subset\ker R_{\infty}$. Taking orthogonal complements,
$\overline{ran\left(R_{\infty}\right)}\subset\overline{ran\left(S\right)}$,
i.e. $s\left(R_{\infty}\right)\le s\left(S\right)$. 
\end{proof}
\begin{prop}[Commuting case]
 \label{prop:5-9} Assume $R_{0}$, $P_{A}$ and $P_{B}$ all commute.
Then 
\[
R_{\infty}=S=R_{0}|_{K}.
\]
Equivalently, on $H_{R_{0}}$ one has $T_{\infty}=P_{M}$. 
\end{prop}

\begin{proof}
If $R_{0}$ commutes with a projection $P$, then $R^{1/2}_{0}$ also
commutes with $P$ by functional calculus. For any positive $R$ that
commutes with $P$, 
\[
\Phi_{P}\left(R\right)=R^{1/2}\left(I-P\right)R^{1/2}=R\left(I-P\right).
\]
Now start from $R_{0}$ and apply the alternating updates using the
commutation $P_{A}P_{B}=P_{B}P_{A}$: 
\begin{align*}
R_{1} & =\Phi_{P_{B}}\left(R_{0}\right)=R_{0}\left(I-P_{B}\right),\\
R_{2} & =\Phi_{P_{A}}\left(R_{1}\right)=R_{1}\left(I-P_{A}\right)=R_{0}\left(I-P_{B}\right)\left(I-P_{A}\right).
\end{align*}
A further step yields 
\begin{align*}
R_{3} & =\Phi_{P_{B}}\left(R_{2}\right)=R_{2}\left(I-P_{B}\right)\\
 & =R_{0}\left(I-P_{B}\right)\left(I-P_{A}\right)\left(I-P_{B}\right)=R_{0}\left(I-P_{B}\right)\left(I-P_{A}\right),
\end{align*}
using idempotence $\left(I-P_{B}\right)^{2}=\left(I-P_{B}\right)$
and commutativity. Thus the sequence stabilizes at 
\[
R_{\infty}=R_{0}\left(I-P_{B}\right)\left(I-P_{A}\right).
\]

For commuting projections, $\left(I-P_{B}\right)\left(I-P_{A}\right)=I-P_{A}-P_{B}+P_{A}P_{B}=P_{K}$,
the orthogonal projector onto $K={\rm ker}P_{A}\cap{\rm ker}P_{B}$.
Hence 
\[
R_{\infty}=R_{0}P_{K}=R_{0}|_{K}=S.
\]
Finally, passing to the intrinsic form via $R^{1/2}_{0}$ gives $T_{\infty}=P_{M}$
on $H_{R_{0}}$. 
\end{proof}

\bibliographystyle{amsalpha}
\bibliography{ref}

\end{document}